\documentclass[11pt,twoside]{amsart}
\usepackage[latin1]{inputenc}
\usepackage[OT1]{fontenc}
\usepackage{amsmath}
\usepackage{amsthm}
\usepackage{amssymb}
\usepackage[all]{xy}
\usepackage{bbm}
\usepackage{amscd}
\usepackage{a4wide}
\usepackage{enumitem}
\usepackage{mathtools}

\setenumerate[0]{label=(\roman*)}

\DeclareMathOperator{\NN}{\mathsf{N}}

\DeclareMathOperator{\cc}{\mathsf{c}}

\DeclareMathOperator{\pr}{\mathsf{pr}}
\DeclareMathOperator{\opr}{\overline{\mathsf{pr}}}
\DeclareMathOperator{\id}{\mathsf{id}}

\DeclareMathOperator{\Coh}{\mathsf{Coh}}

\DeclareMathOperator{\Ho}{\mathsf H}

\DeclareMathOperator{\Tor}{\mathsf{Tor}}

\DeclareMathOperator{\rank}{\mathsf{rank}}

\DeclareMathOperator{\im}{\mathsf{im}}

\let\deg\relax
\DeclareMathOperator{\deg}{\mathsf{deg}}

\newcommand{\IN}{\mathbb{N}}

\newcommand{\wH}{\widetilde H}

\DeclareMathOperator{\VB}{\mathsf{VB}}

\DeclareMathOperator{\UU}{\mathrm{U}}

\makeatletter
\newcommand{\leqnomode}{\tagsleft@true}
\newcommand{\reqnomode}{\tagsleft@false}
\makeatother

\let\ker\relax
\DeclareMathOperator{\ker}{\mathsf{ker}}

\newcommand{\sym}{\mathfrak S}

\newcommand{\reg}{\mathcal O}

\renewcommand{\theta}{\vartheta}
\renewcommand{\rho}{\varrho}
\renewcommand{\phi}{\varphi}

\usepackage[usenames,dvipsnames]{xcolor}
\usepackage{hyperref}
\hypersetup{
    colorlinks,
    linkcolor={red!50!black},
    citecolor={blue!50!black},
    urlcolor={blue!80!black}
}
\usepackage{aliascnt} 

\newtheorem{theorem}{Theorem}[section]

  \newaliascnt{proposition}{theorem}
  \newtheorem{prop}[proposition]{Proposition}
  \aliascntresetthe{proposition}

  \newaliascnt{lemma}{theorem}
  \newtheorem{lemma}[lemma]{Lemma}
  \aliascntresetthe{lemma}

  \newaliascnt{corollary}{theorem}
  
  \aliascntresetthe{corollary}

\theoremstyle{definition}

  \newaliascnt{definition}{theorem}
  
  \aliascntresetthe{definition}

  \newaliascnt{remark}{theorem}
  \newtheorem{remark}[remark]{Remark}
  \aliascntresetthe{remark}

  \newaliascnt{condition}{theorem}
  
  \aliascntresetthe{condition}

  \newaliascnt{question}{theorem}
  
  \aliascntresetthe{question}

  \newaliascnt{example}{theorem}
  
  \aliascntresetthe{example}

\begin{document}

\title{Stability of Tautological Bundles on Symmetric Products of Curves}
\author[A.\ Krug]{Andreas Krug}
\begin{abstract}
We prove that, if $C$ is a smooth projective curve over the complex numbers, and $E$ is a stable vector bundle on $C$ whose slope does not lie in the interval $[-1,n-1]$, then the associated tautological bundle $E^{[n]}$ on the symmetric product $C^{(n)}$ is again stable. Also, if $E$ is semi-stable and its slope does not lie in the interval $(-1,n-1)$, then $E^{[n]}$ is semi-stable.
\end{abstract}

\maketitle

\section*{Introduction}

Given a smooth projective curve $C$ over the complex numbers, there is an interesting series of related higher-dimensional smooth projective varieties, namely the symmetric products $C^{(n)}$. For every vector bundle $E$ on $C$ of rank $r$, there is a naturally associated vector bundle $E^{[n]}$ of rank $rn$ on the symmetric product $C^{(n)}$, called \textit{tautological} or \textit{secant bundle}. These tautological bundles carry important geometric information. For example, $k$-very ampleness of line bundles can be expressed in terms of the associated tautological bundles, and these bundles play an important role in the proof of the gonality conjecture of Ein and Lazarsfeld \cite{Ein-Lazarsfeld--gonalityconj}.
Tautological bundles on symmetric products of curves have been studied since the 1960s \cite{Schwarzenberger--bundlesplane, Schwarzenberger--secant, Mattuck--sym}, but there are still new results about these bundles discovered nowadays; see, for example, \cite{Wang-tautint, MOP, BasuDan-stab}. 

A natural problem is to decide when a tautological bundle is stable. 
Here, stability means slope stability with respect to the ample class $H_n$ that is represented by $C^{(n-1)}+x\subset C^{(n)}$ for any $x\in C$; see \autoref{subsect:intersection} for details. 
This problem has been much studied, mainly in the special case that $L$ is a line bundle; see \cite{AnconaOttaviani--stab, BohnhorstSpindler--stab, MistrettaPhD, EMLN-stab, BiswasNagaraj-stab, DanPal-stab, BasuDan-stab}.
It is easy to see that $E^{[n]}$ can only be stable if $E$ is stable; see \autoref{rem:stabnec}. Hence, the question is under which circumstances the stability of $E$ implies the stability of $E^{[n]}$. 
For line bundles and $n=2$, there is a complete answer given by Biswas and Nagaraj \cite{BiswasNagaraj-stab}. Namely, $L^{[2]}$ is unstable if and only if $L\cong \reg_C$, and $L^{[2]}$ is properly semi-stable if and only if $L\cong \reg(\pm x)$ for some point $x\in C$. Mistretta \cite[Sect.\ 4]{MistrettaPhD}, proved that $L^{[n]}$ is stable whenever $\deg(L)>n-1$. Recently, Dan and Pal \cite{DanPal-stab} and Basu and Dan \cite{BasuDan-stab} started the treatment of the problem for $E$  of higher rank, considering the case $n=2$. In \textit{loc.\ cit.}\ it is shown that $E^{[2]}$ is stable whenever $E$ is stable with $\deg(E)>\rank(E)$, and $E^{[2]}$ is semi-stable whenever $E$ is semi-stable with $\deg(E)\ge\rank(E)$. 

In the present paper, we generalise the result of \textit{loc.\ cit.}\ to arbitrary $n$, and complement it by a similar result for vector bundles of negative degree. Concretely, we prove the following 
\begin{theorem}\label{thm:main}
 Let $C$ be a smooth projective curve, and let $E\in \VB(C)$ be a vector bundle. We set $d:=\deg(E)$, $r:=\rank(E)$, and $\mu:=\mu(E)=\frac dr$.
\begin{enumerate}
\item\label{maini} Let $n\in  \IN$, and let $E$ be semi-stable with $d\ge (n-1)r$ or, equivalently, $\mu\ge n-1$. Then $E^{[n]}$ is slope semi-stable with respect to $H_n$.
\item\label{mainii} Let $n\in  \IN$, and let $E$ be stable with $d> (n-1)r$ or, equivalently, $\mu> n-1$. Then $E^{[n]}$ is slope stable with respect to $H_n$.
\item\label{mainiii} Let $E$ be semi-stable with $d\le -r$ or, equivalently, $\mu\le -1$. Then $E^{[n]}$ is slope semi-stable with respect to $H_n$ for every $n\in \IN$.
\item\label{mainiv} Let $E$ be stable with $d< -r$ or, equivalently, $\mu< -1$. Then $E^{[n]}$ is slope stable with respect to $H_n$ for every $n\in \IN$.
\end{enumerate}
\end{theorem}

The slope of a tautological bundle is given by the formula $\mu(E^{[n]})=\frac{d-(n-1)r}{rn}=\frac{\mu-n+1}n$;
see \autoref{subsec:mutaut}. Hence, we can reformulate our result as follows:

\emph{If the slope $\mu(E^{[n]})$ of a tautological bundle lies outside of the interval $[-1,0]$, the tautological bundle inherits the properties stability and semi-stability from $E$. If $\mu(E^{[n]})$ lies on the boundary of this interval, $E^{[n]}$ still inherits semi-stability from $E$.}

The key to our proof is a short exact sequence relating the tautological bundles $E^{[n-1]}$ and $E^{[n]}$; see \autoref{prop:ses}. This exact sequence allows us to prove \autoref{thm:main}, by a direct argument if $\deg E>0$, and by induction if $\deg E<0$.

The paper is organised as follows. In \autoref{subsec:stability} and \autoref{subsec:taut}, we recall the definitions of slope stability and of tautological bundles on the symmetric product of a curve. In \autoref{subsect:intersection}, we introduce some important divisors on $C^{(n)}$ and $C^n$, and compute their intersection numbers. In \autoref{subsec:stabpull} we show that stability of $E^{[n]}$ can be tested by computing the slopes of $\sym_n$-equivariant subsheaves of $\pi_n^*E^{[n]}$, where $\pi_n\colon C^n\to C^{(n)}$ is the $\sym_n$-quotient morphism. Then, in \autoref{subsec:restrition}, we explain how slopes of $\sym_n$-equivariant sheaves on $C^n$ can be computed by restriction to appropriate subvarieties. In the next \autoref{subsec:pullback}, we discuss the key short exact sequence relating $\pi_n^*E^{[n]}$ and $\pi_{n-1}E^{[n-1]}$. In \autoref{subsec:mutaut}, we compute the slope of tautological bundles and their pull-backs along the quotient morphisms, and remark that (semi-)stability of $E^{[n]}$ implies (semi-)stability of $E$. 
In \autoref{sec:proof}, we carry out the proof of \autoref{thm:main}. Halfway through the proof, we have to separate the cases of negative and positive degree $d$. These two cases are treated in \autoref{subsec:positivedeg} and \autoref{subsec:negativedeg}, respectively. In the final \autoref{sec:sharp}, we observe that \autoref{thm:main} is already optimal in the sense that the numerical conditions on the slopes cannot be weakened.

\subsection*{Conventions}
All our varieties are defined over the complex numbers. We denote the set of positive integers by $\IN$. 
Given two varieties $X$ and $Y$, we write the projections from their product to the factors as
$\pr_X=\pr^{X\times Y}_X\colon X\times Y\to X$ and $\pr_Y=\pr^{X\times Y}_Y\colon X\times Y\to X$.
We write $\VB(X)$ for the category of vector bundles and $\Coh(X)$ for the category of coherent sheaves on $X$.

\subsection*{Acknowledgements}
The author thanks Ben Anthes and S\"onke Rollenske for helpful discussions.

\section{Preliminaries}

\subsection{The notion of slope stability}\label{subsec:stability}
Let $X$ be a smooth projective variety. Let us fix an ample class $H\in \NN^1(X)$ in the group of divisors modulo numerical equivalence. For a coherent sheaf $A\in \Coh(X)$ with $\rank(A)\ge 1$, we define its \textit{degree} and its \textit{slope} with respect to $H$ by
\[
 \deg_H(A):=c_1(A)\cdot H^{n-1}:=\int_X c_1(A)\cdot H^{n-1}\quad,\quad \mu_H(A):=\frac{\deg_H(A)}{\rank(A)}\,.
\]
A vector bundle $E\in \VB(X)$ is called \textit{slope semi-stable} with respect to $H$ if, for every subsheaf $A\subset E$ with $\rank(A)<\rank E$, we have $\mu_H(A)\le \mu_H(E)$. It is called \textit{slope stable} with respect to $H$ if, for every subsheaf $A\subset E$ with $\rank(A)<\rank E$, we have the strict inequality $\mu_H(A)< \mu_H(E)$. Sometimes, we omit the word `slope' and just speak of semi-stable and stable vector bundles. 
Note that, if $X=C$ is a curve, the notion of stability and semi-stability is independent of the chosen ample class $H\in \NN^1(X)$.  

\subsection{Symmetric product of a curve and tautological bundles}\label{subsec:taut}

From now on, let $C$ always be a smooth projective curve, and let $n\in \IN$. There is a natural action by the symmetric group $\sym_n$ on the cartesian product $C^n$ by permutation of the factors. The corresponding quotient variety $C^{(n)}:=C^n/\sym_n$ is called the \textit{$n$-th symmetric product} of $C$. By the Chevalley--Shephard--Todd theorem, the variety $C^{(n)}$ is smooth, and the quotient morphism $\pi_n\colon C^n\to C^{(n)}$ is flat. 

The points of $C^{(n)}$ can be identified with the effective degree $n$ divisors on $C$. Accordingly, we write them as formal sums: $x_1+\dots+x_n=\pi_n(x_1,\dots,x_n)$ for $x_1,\dots, x_n\in C$.
In fact, the symmetric product is the fine moduli space of effective degree $n$ divisors (or, equivalently, zero-dimensional subschemes of length $n$) on $C$, with the universal divisor $\Xi_n\subset C^{(n)}\times C$ given by the image of the closed embedding 
\[C^{(n-1)}\times C\hookrightarrow C^{(n)}\times C\quad,\quad(x_1+\dots+ x_{n-1},x)\mapsto (x_1+\dots+ x_{n-1}+x,x)\,.\]
Now, the Fourier--Mukai transform along this universal divisor allows us to construct tautological vector bundles on $C^{(n)}$ from vector bundles on $C$. Concretely, for $E\in \VB(X)$, the associated \textit{tautological bundle} on $C^{(n)}$ is given by 
\[
 E^{[n]}:=\pr^{C^{(n)}\times C}_{C^{(n)}*}(\reg_{\Xi_n}\otimes \pr^{C^{(n)}\times C*}_{C} E)\cong a_*b^* E \,,
\]
where $a\colon \Xi_n\to C^{(n)}$ and $b\colon \Xi_n\to C$ are the restrictions of the projections $\pr^{C^{(n)}\times C}_{C^{(n)}}$ and $\pr^{C^{(n)}\times C}_{C}$, respectively. Since $a$ is flat and finite of degree $n$, the coherent sheaf $E^{[n]}$ is a vector bundle with $\rank(E^{[n]})=n\rank(E)$.

\subsection{Intersection theory on symmetric and cartesian products of curves}\label{subsect:intersection}

For $i=1,\dots, n$, we write $\pr_i\colon C^n\to C$ for the projection to the $i$-th factor, and $\opr_i\colon C^n\to C^{n-1}$ for the projection to the other $n-1$ factors. 
We set $\wH^n=\sum_{i=1}^n [\pr_i^{-1}(x)]\in \NN^1(C^n)$ for any point $x\in C$. Indeed, modulo numerical equivalence, the divisor $\pr_i^{-1}(x)$ is independent from the point $x\in C$. Using the Segre embedding, we see that $\wH_n$ is ample.

We define $H_n\in \NN^1(C^{(n)})$ as the unique class with $\pi^*H_n=\wH_n$. One can check easily that $H_n$ is represented by $C^{(n-1)}+ x$, the image of the closed embedding $C^{(n-1)}\hookrightarrow C^{(n)}$ with $\alpha\mapsto \alpha +x$, for any $x\in C$. Since $\wH_n$ is ample and $\pi_n\colon C^n\to C^{(n)}$ is finite, $H_n$ is ample too. We always consider stability of bundles on $C^{(n)}$ with respect to this ample class.

Another important divisor on $C^n$ is the \textit{big diagonal} $\delta_n=\sum_{1\le i<j\le n}\Delta_{ij}$ where
\begin{align}\label{eq:pairwise}\Delta_{ij}=\{(x_1,\dots, x_n)\mid x_i=x_j\}\subset C^n\,.\end{align} 
Note that, in the Chow group modulo numerical equivalence, we have 
\begin{align}\label{eq:Hn-1}(\wH_n)^{n-1}=(n-1)!\sum_{i=1}^n \opr_i^{-1}(y)\end{align} for any point $y=(x_1,\dots,x_{n-1})\in C^{n-1}$. Note that $\opr_i^{-1}(y)=x_1\times \dots\times x_{i-1}\times C\times x_{i+1}\times x_{n-1}$. From this, we can easily compute the following intersection numbers
\begin{align}\label{eq:intersectionnumbers}
(\wH_n)^n=n!\quad,\quad \delta_n\cdot (\wH_n)^{n-1}=n!(n-1)\,.
\end{align}

\subsection{Stability under pull-back along quotient morphism}\label{subsec:stabpull} 

Let $G$ be a finite group acting on a smooth projective variety $X$.
A \textit{$G$-equivariant sheaf} on $X$ is a coherent sheaf $B$ together with a $\textit{$G$-linearisation}$, that means a family of isomorphisms $\{\lambda_g\colon B\xrightarrow\sim g^*B\}_{g\in G}$ such that for every pair $g,h\in G$ the following diagram commutes:
\[
\xymatrix{
 B \ar^{\lambda_g}[r]  \ar@/_6mm/^{\lambda_{hg}}[rrr] & g^*B \ar^{g^*\lambda_h}[r] & g^*h^*B\ar^{\cong}[r] & (hg)^*B\,.  
} 
\]  
Let $\pi\colon X\to Y:=X/G$ be the quotient morphism. Then, for every $g\in G$, we have $\pi\circ g=\pi$, which yields a canonical isomorphism of functors $\mu_g\colon \pi^*\xrightarrow\sim g^*\pi^*$. This gives, for every $F\in\Coh(Y)$, a $G$-linearisation $\{\mu_g\colon \pi^*F\xrightarrow g^*\pi^* F\}_{g\in G}$ of $\pi^* F$. We call this the \textit{canonical $G$-linearisation of the pull-back $\pi^*F$}. By a \textit{$\sym_n$-equivariant subsheaf} of $\pi^* F$, we mean a subsheaf $A\subset \pi^*F$ which is preserved by the canonical $G$-linearisation of the pull-back: $\mu_g(A)=g^*A$ as subsheaves of $g^*\pi^*F$. 

\begin{lemma}\label{lem:stabpullbackequi}
Let a finite group $G$ act on a smooth projective variety $X$ such that $Y=X/G$ is again smooth and $\pi\colon X\to Y$ is flat. Let $H\in \NN^1(Y)$ be an ample class and $F\in \VB(Y)$.
 \begin{enumerate}
\item\label{i} If $\mu_{\pi^*H}(A)\le \mu_{\pi^*H}(\pi_n^*F)$ holds for all $\sym_n$-equivariant subsheaves $A$ of $\pi^*F$ with $\rank A<\rank F$, then $F$ is slope semi-stable with respect to $H$. 
\item\label{ii} If $\mu_{\pi^*H}(A)< \mu_{\pi^*H}(\pi_n^*F)$ holds for all $\sym_n$-equivariant subsheaves $A$ of $\pi^*F$ with $\rank A<\rank F$, then $F$ is slope stable with respect to $H$. 
\end{enumerate}
\end{lemma}
\begin{proof}
For every $B\in \Coh(Y)$ with $\rank B>0$, we have $\mu_{\pi^*H}(\pi^*B)=|G|\cdot \mu_H(B)$; see \cite[Lem.\ 3.2.1]{HL}. Hence, for $F$ to be semi-stable, it is sufficient to have $\mu_{\pi^*H}(\pi^*B)\le \mu_{\pi^*H}(\pi^*F)$ for every subsheaf $B\subset F$ with $\rank B<\rank F$. The assertion of part \ref{i} follows from the fact that $\pi^*B$ is a $G$-equivariant subsheaf of $\pi^*F$ whenever $B$ is a subsheaf of $F$. The proof of part \ref{ii} is completely analogous.   
\end{proof}

See \cite[Sect.\ 4.2]{MistrettaPhD} for a similar criterion for slope stability of sheaves on quotients. 

\subsection{Some technical lemmas concerning restriction of sheaves}\label{subsec:restrition}

\begin{lemma}\label{lem:torvanish}
Let $X$ be a smooth variety, $D\subset X$ an effective divisor, and $F\in \Coh(X)$. Then:
\[
 \Tor_1(\reg_D,F)=0\quad\Longleftrightarrow\quad \text{$D$ does not contain an associated point of $F$.}
\]
\end{lemma}
\begin{proof}
Let $s\in \Ho^0(\reg_X(D))$ be a global section defining $D$. Since 
\[0\to \reg_X(-D)\xrightarrow{s} \reg_X\to \reg_D\to 0\]
is a locally free resolution of $\reg_D$, we have $\Tor_1(\reg_D,F)=\ker(s\otimes \id_F)$. This kernel is non-trivial if and only if $D$ contains an associated point of $F$. 
\end{proof}

\begin{lemma}\label{lem:torvanishCn}
Let $F_1,\dots,F_k\in \Coh(C^n)$ be finite collection of coherent sheaves on $C^n$, and let $m\le n$. Then there exist points $x_{m+1},\dots,x_n\in C$ such that, if 
\[
 \iota\colon C^m\hookrightarrow C^n\quad, \quad (t_1,\dots, t_m)\mapsto (t_1,\dots, t_m, x_{m+1},\dots, x_n)
\]
denotes the closed embedding with image $C^m\times x_{m+1}\times \dots\times x_n$, the ranks of the $F_i$ do not change under pull-back along $\iota$ and all the higher pull-backs of the $F_i$ along $\iota$ vanish:
\[
\forall\,i=1,\dots,k\,: \quad \rank(\iota^*F_i)=\rank(F_i)\quad, \quad L^{j}\iota^* (F_i)=0 \quad\text{for $j\neq 0$.} 
\]
\end{lemma}
\begin{proof}
We proceed by inverse induction on $m$. For $m=n$, we have $\iota=\id$, and the assertion is trivial. For general $m$, by the induction hypothesis, we may assume that we already found $x_{m+2},\dots, x_n\in C$ such that, for 
\[
 \alpha\colon C^{m+1}\hookrightarrow C^n\quad,\quad (t_1,\dots, t_{m+1})\mapsto (t_1,\dots, t_{m+1}, x_{m+2}, \dots, x_n)\,, 
\]
we have $\rank(\alpha^*F_i)=\rank(F_i)$ and $L^{j}\alpha^* (F_i)=0$ for all $i=1,\dots, k$ and all $j\neq 0$. 
Now, there exists a non-empty $U\subset C^{m+1}$ over which all the $\alpha^*F_{i}$ are locally free. We choose $x_{m+1}\in C$ in such a way that $C^m\times x_{m+1}\subset C^{m+1}$ has a non-empty intersection with $U$ and does not contain any of the finitely many associated points of the sheaves $\alpha^*F_{i}$. 
We have $\iota=\beta\circ \alpha$ where $\beta\colon C^m\to C^{m+1}$ with $\beta(t_1,\dots, t_m)=(t_1,\dots, t_m,x)$ is the closed embedding with image $C^m\times x_{m+1}$.  
Because of 
$C^m\times x_{m+1}$ meeting $U$, we have
\[\rank(\iota^*F_i)=\rank(\alpha^* F_i)=\rank(F_i) \quad \text{for every $i=1,\dots,k$.}\]
By our choice to avoid the associated points of the $\alpha^*F_i$, \autoref{lem:torvanish} gives $L^{p}\beta^*(\alpha^* F_i)=0$ for all $p\neq 0$. The assertion follows by the spectral sequence
\[
E_2^{p,q}=L^{p}\beta^*(L^q \alpha^* (F_i))\quad\Longrightarrow \quad E^j=L^j \alpha^* (F_i)\,. \qedhere 
\]
\end{proof}

\begin{lemma}\label{lem:restrictiondeg}
 Let $A$ be an $\sym_n$-equivariant sheaf on $C^n$.
\begin{enumerate}
 \item\label{1}  Let $x_{2},\dots,x_n\in C$ be points such that $L^j\iota^*(A)=0$ for all $j\neq 0$ where 
$\iota\colon C\hookrightarrow C^n$ is given by $\iota(t)=(t, x_2,\dots, x_n)$. Then 
 \[
  \deg_{\wH_n}(A)=n!\deg(\iota^*A)\,.
 \]
 \item\label{2}  Let $x\in C$ be a point such that $L^j\iota^*(A)=0$ for all $j\neq 0$ where 
$\iota\colon C^{m-1}\hookrightarrow C^n$ is given by $\iota(t_1,\dots,t_{n-1})=(t_1,\dots,t_{n-1},x)$. Then 
 \[
  \deg_{\wH_n}(A)=n\deg_{\wH_{n-1}}(\iota^*A)\,.
 \]
\end{enumerate}
\end{lemma}

\begin{proof}
In the set-up of part \ref{1}, we have $\iota(C)=\opr_1^{-1}(y)$ with $y=(x_2,\dots, x_n)\in C^{n-1}$. By projection formula, we have $[\opr_1^{-1}(y)]\cdot \cc_1(A)=\iota^*\cc_1(A)=\cc_1(\iota^*A)=\deg(\iota^*A)$, where the equality $\iota^*\cc_1(A)=\cc_1(\iota^*A)$ is due to the vanishing of the higher derived pull-backs. By the $\sym_n$-equivariance of $A$, we get $[\opr_i^{-1}(y)]\cdot \cc_1(A)=[\opr_1^{-1}(y)]\cdot \cc_1(A)=\deg(\iota^*A)$ for every $i=1,\dots, n$. Combining this with \eqref{eq:Hn-1} gives
\begin{align*}
 \deg_{\wH_n}(A)=(\wH_n)^{n-1}\cdot \cc_1(A)=(n-1)!\bigl(\sum_{i=1}^n [\opr_i^{-1}(y)]\bigr)\cdot \cc_1(A)&=(n-1)!n[\opr_1^{-1}(y)]\cdot \cc_1(A)\\&=n!\deg(\iota^*A)\,.
\end{align*}
The proof of part \ref{2} is very similar.
\end{proof}

\subsection{Pull-back of tautological bundles along the $\sym_n$-quotient}\label{subsec:pullback}

For $i=1,\dots,n$, we consider the divisor $\delta_n(i):=\sum_{j\in \{1,\dots,n\}\setminus\{i\}} \Delta_{ij}$ on $C^n$; compare \eqref{eq:pairwise}.

\begin{prop}\label{prop:ses}
For every $i=1,\dots, n$, there is a short exact sequence
\begin{align}\label{eq:keyses}
0\to \pr_i^* E(-\delta_n(i))\to \pi_n^* E^{[n]}\to \overline\pr_{i}^*\pi_{n-1}^*E^{[n-1]}\to 0\,. 
\end{align}
The subsheaves $\UU_n(E,i):=\im\bigl(\pr_i^* E(-\delta_n(i))\to \pi_n^* E^{[n]}\bigr)$ of $\pi_n^*E^{[n]}$ defined by these sequences have pairwise trivial intersections:
\begin{align}\label{eq:pairwisedisj}
\UU_n(E,i)\cap \UU_n(E,i)\quad\text{for $i\neq j$.} \end{align}
Furthermore, these subsheaves get permuted by the natural $\sym_n$-linearisation of the pull-back $\pi_n^* E$:
If $\sigma(i)=j$, we have the equality $\mu_\sigma(\UU_n(E,i))=\sigma^*\UU_n(E,j)$ of subsheaves of $\sigma^*\pi_n^*E^{[n]}$.
\end{prop}
\begin{proof}
The pull-back of the universal divisor $\Xi_n\subset C^n\times C$ along the flat morphism $\pi_n\times \id_C\colon C^n\times C\to C^{(n)}\times C$ is given by $(\pi_n\times\id_C)^*\Xi_n=D_n:=\sum_{k=1}^n \Gamma_{\pr_k}$.
By flat base change along the diagram 
\[
\xymatrix{
 C^n\times C \ar^{\pi_n\times \id_C}[r]  \ar_{\pr_{C^n}}[d] & C^{(n)}\times C \ar^{\pr_{C^{(n)}}}[d] \ar^{\quad\pr_C}[r]& C   \\
 C^n \ar^{\pi_n}[r] & C^{(n)} \,, & 
}
\]
we get, setting $q_n:=\pr^{C^n\times C}_C=\pr^{C^{(n)}\times C}_C\circ (\pi_n\times \id_C)\colon C^n\times C\to C$, the following isomorphism 
\begin{align}\label{eq:tautFM}\pi_n^*E^{[n]}\cong \pr_{C^n*}(\reg_{D_n}\otimes q_n^* E)\,.\end{align} 
Now, let us fix some $i\in [n]$. 
We note that $\sum_{k\neq i} \Gamma_{\pr_k}=(\overline\pr_i\times \id_C)^* D_{n-1}$, which gives $D_n=\Gamma_{\pr_i} + (\overline\pr_1\times \id_C)^* D_{n-1}$. Hence, we get a short exact sequence
\begin{align}\label{eq:Dses}
 0\to \reg_{\Gamma_{\pr_i}}(-\sum_{k\neq i} [\Gamma_{\pr_k}\cap \Gamma_{\pr_i}])\to \reg_{D_n}\to (\overline\pr_i\times\id_C)^*\reg_{D_{n-1}}\to 0
\end{align}
of coherent sheaves on $C^n\times C$. All the sheaves of this sequence are finitely supported over $C^n$. Hence, combining \eqref{eq:tautFM} and \eqref{eq:Dses}, gives the short exact sequence
\begin{align}\label{eq:Eses}
 0\to \pr_{C^n*}\bigl(\reg_{\Gamma_{\pr_i}}(-\sum_{k\neq i} [\Gamma_{\pr_k}\cap \Gamma_{\pr_i}])\otimes q_n^*E\bigr)\to \pi_n^* E^{[n]}\to \pr_{C^n*}\bigl((\overline\pr_i\times\id_C)^*\reg_{D_{n-1}}\otimes q_n^*E\bigr)\to 0
\end{align}
of coherent sheaves on $C^n$, which will turn out to be isomorphic to the asserted sequence \eqref{eq:keyses}. 
By flat base change along the diagram 
\[
\xymatrix{
 C^n\times C \ar^{\overline\pr_i \times \id_C}[r]  \ar_{\pr_{C^n}}[d] \ar@/^8mm/^{q_n}[rr] & C^{n-1}\times C \ar^{\pr_{C^{n-1}}}[d] \ar^{\quad q_{n-1}}[r]& C   \\
 C^n \ar^{\overline\pr_1}[r] & C^{n-1} \,, & 
}
\]
we see that $\pr_{C^n*}\bigl((\overline\pr_i\times\id_C)^*\reg_{D_{n-1}}\otimes q_n^*E\bigr)\cong \overline\pr_{i}^*\pi_{n-1}^*E^{[n-1]}$. To bring the first term of \eqref{eq:Eses} into the correct form, we consider the isomorphism 
\[t\colon C^n\xrightarrow\cong \Gamma_{\pr_1}\subset C^n\times C\quad,\quad (x_1,\dots, x_n)\mapsto (x_1,\dots, x_n;x_i)\,.\] 
Because of 
$\Gamma_{\pr_k}\cap \Gamma_{\pr_i}=\{(x_1,\dots,x_n;x)\in C^n\times C\mid x_k=x_i=x\}$,
we see that \[t^*\bigl(\sum_{k\neq i} [\Gamma_{\pr_k}\cap \Gamma_{\pr_i}]\bigr)=\delta_n(i)\,.\]
From this, it follows that 
$\pr_{C^n*}\bigl(\reg_{\Gamma_{\pr_i}}(-\sum_{k\neq i} [\Gamma_{\pr_i}\cap \Gamma_{\pr_i}])\cong \pr_i^*E(-\delta_n(i))$, which shows that the sheaves in \eqref{eq:Eses} are isomorphic to those in \eqref{eq:keyses}. 

The fact that, for $i\neq j$, the subsheaves $\pr_i^*E(-\delta_n(i))$ and $\pr_j^*E(-\delta_n(j))$ of $\pi_n^*E^{[n]}$ intersect trivially follows from the fact that $\reg_{\Gamma_{\pr_i}}(-\sum_{k\neq i} [\Gamma_{\pr_k}\cap \Gamma_{\pr_i}])$ and $\reg_{\Gamma_{\pr_j}}(-\sum_{k\neq j} [\Gamma_{\pr_k}\cap \Gamma_{\pr_j}])$ intersect trivially as subsheaves of $\reg_D$.

The final statement of the proposition follows from the fact that, for $\sigma\in \sym_n$ with $\sigma(i)=j$, we have the equality
\[\nu_\sigma(\reg_{\Gamma_{\pr_i}}(-\sum_{k\neq i} [\Gamma_{\pr_k}\cap \Gamma_{\pr_i}])=\sigma^*\reg_{\Gamma_{\pr_j}}(-\sum_{k\neq i} [\Gamma_{\pr_k}\cap \Gamma_{\pr_i}])\]
of subsheaves of $\sigma^*\reg_{D_n}$, where $\nu$ is the natural $\sym_n$-linearisation of the pull-back $\reg_{D_n}=(\pi_n^*\times \id_C)^*\reg_{\Xi_n}$.
\end{proof}

\subsection{Degree and slope of tautological bundles}\label{subsec:mutaut}

There are well-known formulae for the Chern classes of tautological bundles; see \cite[Sect.\ 3]{Mattuck--sym}. In particular, we have
\begin{align}\label{eq:c1}
 \cc_1(\pi_n^* E^{[n]})= d \wH_n -r \delta_n\,.
\end{align}
Alternatively, this formula can easily be deduced inductively using the short exact sequence of \autoref{prop:ses}. For doing this, note that $\wH_n=\pr_i^*[x] + \overline\pr_i^*H_{n-1}$ and $\delta_n=\delta_n(i)+ \overline\pr_i^* \delta_{n-1}$ for every $i\in [n]$. Combining \eqref{eq:c1} with \eqref{eq:intersectionnumbers}, we get $\deg_{\wH_n}(\pi_n^* E^{[n]})= n!\bigl(d-(n-1)r\bigr)$ and 
\begin{align}\label{eq:mutaut}
 \mu_{\wH_n}(\pi^*_nE^{[n]})=\frac{(n-1)!(d-(n-1)r)}r=(n-1)!(\mu-n+1)\,.
\end{align}  
Since $\pi_n$ is finite of degree $n!$, we also get
\begin{align}\label{eq:mutautsym}
\mu_{H_n}(E^{[n]})=\frac{\mu_{\wH_n}(\pi^*_nE^{[n]})}{n!}=\frac{(d-(n-1)r)}{nr}=\frac{\mu-n+1}n 
\end{align}
\begin{remark}\label{rem:stabnec}
For an arbitrary, not necessarily locally free, coherent sheaf $A\in \Coh(C)$, we can still define an associated tautological sheaf on $C^{(n)}$ by $A^{[n]}:=a_*b^*(A)$; compare \autoref{subsec:taut}. Since $a$ is finite and $b$ is flat, the functor $a_*b^*\colon \Coh(C)\to \Coh(C^{(n)})$ is exact; compare \cite[Thm.\ 1.1]{Krug--reconstruction}. In particular, if $0\to E_1\to E_0\to A\to 0$ is a locally free resolution of $A$, then $0\to E_1^{[n]}\to E_0^{[n]}\to A^{[n]}\to 0$ is a locally free resolution of $A^{[n]}$. It follows that formula \eqref{eq:mutautsym} extends to a formula for slopes of tautological sheaves of positive rank, namely $\mu_{H_n}(E^{[n]})=\frac{\mu(A)-n+1}n$. It follows that, if $E$ is a vector bundle on $C$ and $A\subset E$ is a destabilising sheaf, then $A^{[n]}\subset E^{[n]}$ is again destabilising. In other words, (semi-)stability of $E^{[n]}$ implies (semi-)stability of $E$. 
\end{remark}

\section{Proof of the main result}\label{sec:proof}

\subsection{General part of the proof}\label{subsec:generalproof}

Let $E\in \VB(C)$ satisfy the assumptions of one of the four parts \ref{maini}, \ref{mainii}, \ref{mainiii}, \ref{mainiv} of \autoref{thm:main}. By \autoref{lem:stabpullbackequi}, in order to proof stability or semi-stability of $E^{[n]}$, we need to compare the slopes of $A$ and $\pi_n^*E^{[n]}$ for $A\subset \pi_n^*E^{[n]}$  a $\sym_n$-invariant subsheaf 
with $s:=\rank A<nr=\rank E^{[n]}$. 

For $i=1,\dots, n$, we set $A'(i):=A\cap \UU_n(E, 1)$ as an intersection of subsheaves of $\pi_n^* E^{[n]}$; compare \autoref{prop:ses}. We write the corresponding quotient as $A''(i)=A/A'(i)$. We also set $A'=A'(1)$ and $A''=A''(1)$, and get a commutative diagram with exact columns and rows
\begin{align}\label{diag:full}
\xymatrix{
 & 0 \ar[d] & 0 \ar[d] & 0 \ar[d] &\\
0\ar[r] &  A' \ar[r]  \ar[d] & A\ar[r]\ar[d] &  A''\ar[r]\ar[d]& 0\\
0\ar[r] & \pr_1^* E(-\delta_n(1))\ar[r] & \pi_n^* E^{[n]}\ar[r]& \overline \pr_{1}^*\pi_{n-1}^*E^{[n-1]}\ar[r] & 0}
\end{align}
where the bottom row is the short exact sequence form \autoref{prop:ses}.
We set $s'=\rank A'$ and $s''=\rank A''$ which gives $s=s'+s''$. By the last statement of \autoref{prop:ses} together with the $\sym_n$-equivariance of the subsheaf $A\subset \pi_n^*E^{[n]}$, we have $A'(i)\cong\sigma^*A'$ for any $\sigma\in \sym_n$ with $\sigma(i)=1$. In particular, $\rank A'(j)=s'$ for every $j=1,\dots,n$. By \eqref{eq:pairwisedisj}, we have $\bigoplus_{j=1}^nA'(j)\subset A$. Hence, we get the following inequalities of the ranks
\begin{align}\label{eq:sss}
 s\ge ns'\quad, \quad s''\ge (n-1)s'\quad,\quad ns''\ge (n-1)s\,. 
\end{align}

Now, we divide the proof that $\mu_{\wH_n}(A)\le\mu_{\wH_n}(\pi_n^*E^{[n]})$ (or, for the proof of parts \ref{mainii} and \ref{mainiv} of \autoref{thm:main}, that we have have a strict inequality) into the two cases of positive and negative $d=\deg E$, treated in the following two subsections.  

\subsection{Proof of the main theorem for bundles of positive degree}\label{subsec:positivedeg}

In this subsection, we proof parts \ref{maini} and \ref{mainii} of \autoref{thm:main}. Let $E\in \VB(C)$ be a semi-stable bundle with $d\ge (n-1)r$, equivalently $\mu\ge (n-1)$. By \autoref{lem:torvanishCn} and \autoref{lem:restrictiondeg}, there are points $x_2,\dots, x_n\in C$ such that, for $\iota\colon C\hookrightarrow C^n$ with $\iota(t)=(t,x_2,\dots,x_n)$ the closed embedding with image $C\times x_2\times \dots\times x_n$, we have \begin{align}\label{eq:degrespositive}\deg_{\wH_n}(A)=n!\deg(\iota^*A)\,,\end{align}
the rank of objects of diagram \eqref{diag:full} remain unchanged after pull-back by $\iota$, and the rows and columns of the diagram \eqref{diag:full} remain exact after pull-back by $\iota$. Since $C\times x_2\times\dots\times x_n$ is a section of $\pr_1\colon C^n\to C$ and a fibre of $\overline \pr_1\colon C^n\to C^{n-1}$, the restricted diagram takes the form 
\begin{align}\label{diag:res1}
\xymatrix{
 & 0 \ar[d] & 0 \ar[d] & 0 \ar[d] &\\
0\ar[r] &  \iota^*A' \ar[r]  \ar[d] &\iota^*A\ar[r]\ar[d] & \iota^*A''\ar[r]\ar[d]& 0\\
0\ar[r] & E(-x_2-x_3-\dots-x_n) \ar[r] & \iota^*\pi_n^* E^{[n]}\ar[r]& \reg_C^{\oplus r(n-1)}\ar[r] & 0}
\end{align}
By the semi-stability of $E(-x_2-x_3-\dots-x_n)$ and $\reg_C^{\oplus r(n-1)}$, we get
\begin{align}\label{ineq1}\mu(\iota^* A')\le \mu(E(-x_2-x_3-\dots-x_n))=\mu-n+1\end{align}
and $\mu(\iota^*A'')\le 0$. Hence,
\begin{align}\label{ineq2}
 \deg(\iota^*A)=\deg(\iota^*A')+\deg(\iota^*A'')= s'\mu(\iota^*A')+s''\mu(\iota^*A'')\le s'(\mu-n+1)\,.
\end{align}
By \eqref{eq:degrespositive}, and by the inequality $ns'\le s$ of \eqref{eq:sss} combined with the assumption $\mu\ge n-1$,  
\begin{align}\label{ineq3}
 \mu_{\wH_n}(A)=n!\frac{\deg(\iota^*A)}s\le n! \frac{s'}s(\mu-n+1)\le (n-1)!(\mu-n+1)=\mu(\pi_n^*E^{[n]})\,.
\end{align}
By \autoref{lem:stabpullbackequi}, this shows that $E^{[n]}$ is semi-stable. 

Let now $E$ be stable and $\mu>n-1$. Then, by the stability of $E(-x_2-x_3-\dots-x_n)$,
the inequality \eqref{ineq1} is strict. Accordingly, the inequality in \eqref{ineq2} and the first inequality in \eqref{ineq3} are strict, except for if $s'=0$. However, for $s'=0$ the second inequality of \eqref{ineq3} is strict, due to the assumption $\mu>n-1$. Hence, in any case, we have $\mu_{\wH_n}(A)<\mu_{\wH_n}(\pi_n^*E^{[n]})$ so that $E^{[n]}$ is stable by \autoref{lem:stabpullbackequi}.

\subsection{Proof of the main theorem for bundles of negative degree}\label{subsec:negativedeg}
In this subsection, we prove part \ref{mainiii} and \ref{mainiv} of \autoref{thm:main}. So, let $E\in \VB(C)$ be a semi-stable bundle with $\mu\le -1$. 
We argue by induction on $n$ that $\mu_{\wH_n}(A)\le \mu_{\wH_n}(\pi_n^*E)$ for every $\sym_n$-equivariant subsheaf with $s:=\rank A<nr=\rank E^{[n]}$. For $n=1$, the assertion is trivial as $\pi_1^*E^{[1]}=E$. 
Let now $n\ge 2$.
By \autoref{lem:torvanishCn} and \autoref{lem:restrictiondeg}, there is an $x\in C$ such that, for $\iota\colon C^{n-1} \hookrightarrow C^n$ with $\iota(t_2,\dots,t_n)=(x,t_2,\dots,t_n)$ the closed embedding with image $x\times C^{n-1}$, we have 
\begin{align}\label{eq:degres}\deg_{\wH_n}(A)=n\deg_{\wH_{n-1}}(\iota^*A)\,,\end{align}
the rank of objects of diagram \eqref{diag:full} remain unchanged after pull-back by $\iota$, and the rows and columns of the diagram \eqref{diag:full} remain exact after pull-back by $\iota$.
Noting that $\iota^*(\delta_n(1))=\wH_{n-1}$ and $\opr_1\circ \iota=\id_{C^{n-1}}$, the restricted diagram takes the form
\begin{align}\label{diag:res2}
\xymatrix{
 & 0 \ar[d] & 0 \ar[d] & 0 \ar[d] &\\
0\ar[r] &  \iota^*A' \ar[r]  \ar[d] &\iota^*A\ar[r]\ar[d] & \iota^*A''\ar[r]\ar[d]& 0\\
0\ar[r] & \reg(-\wH_{n-1})^{\oplus r}\ar[r] & \iota^*\pi_n^* E^{[n]}\ar[r]& \pi_{n-1}^*E^{[n-1]}\ar[r] & 0}
\end{align}
By the induction hypothesis, together with \eqref{eq:mutaut}, we get 
\begin{align}\label{eq:muA''}
\mu_{\wH_{n-1}}(\iota^*A'')\le \mu_{\wH_{n-1}}(\pi_{n-1}^*E^{[n-1]})= (n-2)!(\mu-n+2)\,. 
\end{align}
Furthermore, the inclusion $\iota^*A'\hookrightarrow \reg(-\wH_{n-1})^{\oplus r}$ combined with \eqref{eq:intersectionnumbers} gives
\begin{align}\label{eq:muA'}
 \mu_{\wH_{n-1}}(\iota^*A')\le \mu_{\wH_{n-1}}(\reg(-\wH_{n-1})^{\oplus r})=-(n-1)!\,.
\end{align}
Combining \eqref{eq:muA''} and \eqref{eq:muA'}, we get 
\begin{align}
 \deg_{\wH_{n-1}}(\iota^*A)&=s''\mu_{\wH_{n-1}}(\iota^*A'')+s'\mu_{\wH_{n-1}}(\iota^*A')\notag\\
&\le (n-2)!\bigl(s''(\mu-n+2)-s'(n-1) \bigr)\,. \label{eq:upperbound}
\end{align}
By the assumption $\mu\le -1$, we have $\mu-n+2\le -(n-1)$. Hence, \eqref{eq:upperbound} is maximised if the inequality $s''\ge s'(n-1)$ from \eqref{eq:sss} is an equality. This gives
\[
\deg_{\wH_{n-1}}(\iota^*A)\le (n-2)!\bigl(s''(\mu-n+2)-s'' \bigr)=s''(n-2)!(\mu-n+1) 
\]
We get the following chain of inequalities
\begin{align}\label{eq:finalchain}
 \mu_{\wH_n}(A)=\frac{\deg_{\wH_n}(A)}s\le \frac{ns''(n-2)!(\mu-n+1)}{s}\le (n-1)!(\mu-n+1) =\mu_{\wH_n}(\pi_n^*E^{[n]})\,,
\end{align}
where the first inequality is due to \eqref{eq:degres}, the second is due to the inequality $ns''\ge (n-1)s$ of \eqref{eq:sss} together with the fact that $\mu-n+1$ is non-positive, and the last equality is \eqref{eq:mutaut}.
This proves that $E^{[n]}$ is semi-stable.

Let now $E$ be stable and $\mu<-1$. Proceeding as before by induction, we see that, if $s''<(n-1)r$, the inequality \eqref{eq:muA''} is strict. Hence, the first inequality of \eqref{eq:finalchain} is strict too. If $s''=(n-1)r$, we have $ns''>(n-1)s$ since $s<nr$. It follows that, in this case, the second inequality of \eqref{eq:finalchain} is strict.

\section{The numerical conditions are sharp}\label{sec:sharp}

In this section, we observe that the numerical conditions in \autoref{thm:main} on the slope cannot be weakened. For this, we consider examples of (semi)-stable bundles on $C$ with various values $\mu(E)\in [-1,n-1]$ such that $E^{[n]}$ is unstable.

Let $\ell\ge 0$, $x\in C$, and $L=\reg_C(\ell\cdot x)$. Any non-zero section of $L$ induces a non-zero section of $L^{[n]}$; see \cite[Corollary of Prop.\ 1]{Mattuck--sym}. Hence, $\reg_{X^{(n)}}$ is a subsheaf of $L^{[n]}$. For $0\le \ell< n-1$, we have $\mu(L^{[n]})<0$; see \eqref{eq:mutautsym}. Hence, in this case, the subsheaf $\reg_{X^{(n)}}$ is destabilising. For $\ell=n-1$, we have $\mu(L^{[n]})=0$ and $L^{[n]}$ is properly semi-stable.

In a similar way, we get examples of higher rank and non-integer slope: Whenever $E\in \VB(C)$ has $\mu(E)<n-1$ and $h^0(E)>0$, the structure sheaf $\reg_{C^{(n)}}$ is a destabilising subsheaf of $E^{[n]}$. For many curves $C$ and many values of $d$ and $r$ such that $\mu(E)<n-1$, the existence of stable bundles with $h^0(E)>0$ is guaranteed by Brill--Noether theory.  

The tautological bundles $L^{[n]}$ associated to $L=\reg(-x)$, which have slope $\mu(L^{[n]})=-1$ for every $n\in\IN$, can also be shown to be properly semi-stable as follows. We consider the bundle $L^{\boxplus n}:=\bigoplus_{i=1}^n\pr_i^*L$ on $C^n$ equipped with the $\sym_n$-linearisation given by permutation of the direct summands. We have an isomorphism $L^{[n]}\cong \pi_{n*}^{\sym_n}L^{\boxplus n}$, where $\pi_{n*}^{\sym_n} L^{\boxplus n}$ are the invariants of $\pi_{n*}L^{\boxplus n}$ under the $\sym_n$-linearisation. Every morphism $s\colon L\hookrightarrow \reg_C$ induces an $\sym_n$-equivariant embedding $L^{\boxtimes n}:=\bigotimes_{i=1}^n\pr_i^* E\hookrightarrow L^{\boxplus n}$ with components \[s^{\boxtimes i-1}\boxtimes \id\boxtimes s^{\boxtimes n-i}\colon L^{\boxtimes n}\to\pr_i^*L=\reg_C^{\boxtimes i-1}\boxtimes L\boxtimes \reg_C^{\boxtimes n-1}\,. \] 
Since $\pi_{n*}^{\sym_n}$ is exact, we have an inclusion $\pi_{n*}^{\sym_n} L^{\boxtimes n} \hookrightarrow L^{[n]}$. Furthermore, $\pi_n^*\pi_{n*}^{\sym_n} L^{\boxtimes n}\cong L^{\boxtimes n}$. Hence, 
\[
 \mu_{H_n}(\pi_{n*}^{\sym_n}L^{\boxtimes n})=\frac{\mu_{\wH_n}(L^{\boxtimes n})}{n!}=-1= \mu_{H_n}(L^{[n]})\,,
\]
which shows that $L^{[n]}$ is properly semi-stable.

Note however, that it is still possible that there are stable tautological bundles with slope lying in the interval $[-1,0]$. At least, there are stable tautological bundles on the boundary of this interval in the case $n=2$: If $L$ is of degree $1$ but not isomorphic to $\reg_C(x)$ for any $x\in C$, or of degree $-1$ but not isomorphic to $\reg_C(-x)$ for any $x\in C$, the tautological bundle $L^{[2]}$ is stable (not only semi-stable) of slope $-1$ or $0$; see \cite{BiswasNagaraj-stab}.

\bibliographystyle{alpha}
\addcontentsline{toc}{chapter}{References}
\bibliography{references}

\end{document}